\documentclass[a4paper,12pt]{article}
\usepackage[utf8]{inputenc}
\usepackage{epsfig}
\usepackage{float}
\usepackage{graphicx}
\usepackage{graphics}
\usepackage{amssymb}
\usepackage{amsfonts}
\usepackage{newlfont}
\usepackage{indentfirst}
\usepackage{amsmath}
\usepackage{mathrsfs}
\usepackage{cases}
\usepackage{txfonts}
\usepackage{amsfonts}
\usepackage{amsmath,amssymb}
\usepackage{lipsum}
\usepackage{ntheorem}

\newtheorem{theorem}{Theorem}[section]
\newtheorem{lemma}{Lemma}[section]
\newtheorem*{lemmaA}{Lemma}

\newtheorem{remark}{Remark}[section]

\numberwithin{equation}{section}
\newenvironment{proof}[1][Proof]{\noindent\textbf{#1.} }{\hfill $\Box$}

\begin{document}
\date{}
\title{\bf{Optimal control of an age-structured problem modelling mosquito plasticity}}
\author{Lin Lin Li \textsuperscript{a}, Cl\'audia Pio Ferreira \textsuperscript{b}, Bedreddine Ainseba \textsuperscript{a}\\
\\
\footnotesize{\textsuperscript{a} Institut de Math\'ematiques de Bordeaux, Universite de Bordeaux, Bordeaux, France}\\
\footnotesize{\textsuperscript{b} Departamento de Bioestat\'{\i}stica, Intituto de Bioci\^encias de Botucatu/Unesp}}
\maketitle

\begin{abstract}
\noindent{In this paper, we study an age-structured model which has strong biological background about mosquito plasticity. Firstly, we prove the existence of solutions and the comparison principle for a generalized system. Then, we prove the existence of the optimal control for the best harvesting. Finally, we establish necessary optimality conditions.}
\end{abstract}

\section{Introduction}
\noindent
Throughout the human history, people have always been combating against many infectious diseases, such as malaria, dengue,
yellow and Chikungunya fever, encephalitis and the diseases have caused uncounted mortality of mankind. During the past decades, many researchers studied the pathology of these infectious diseases and tried to control the transmission of them.
One of the most studied diseases is malaria, which is mainly transmitted by Anopheles gambiae and Anopheles funestus, the main vectors \cite{Bn}. As the statistical data show, malaria affects more than 100 tropical countries, placing 3.3
billion people at risk \cite{WHO1} and the life of one African child's life is taken by malaria every minute \cite{WHO2}. To reduce human's suffering from malaria, people have been seeking efficient ways to control the malaria transmission for many years. In the past decades, the control of malaria has made slow but steady progress and the overall mortality rate has dropped by more than $25\%$ since 2000 \cite{Ma}. The main strategies of controlling malaria are insecticide treated nets (ITNs) and indoor residual spraying (IRS) \cite{Bn,Hb,Ma,WHO1}. However, the effectiveness of these strategies depends on the susceptibility of the vector species to insecticides and their behaviours, ecology and population genetics.

ITNs and IRS are efficient ways against the main vectors of malaria in Africa. However, the resistance of mosquitoes to insecticides forces them to adapt their behaviours to ensure their survival and reproduction. Especially, they can adapt their bitting behaviour from night to daylight \cite{Ss}. This new behavioural patterns lead to a resurgence of malaria morbidity in several parts of Africa \cite{Tj}. Thus, new methods are desired to replace the traditional strategies.

In this work, we are going to model mosquito population adaption and study the optimal control problem. We consider a linear model describing the dynamics of a single species population with age dependence and spatial structure as follows
\begin{equation}\label{eq:01}
\left\{ \begin{array}{lll}
Dp- \delta \Delta p+\mu (a)p = u(a,t,x)p,&(a,t,x)\in Q_{a_\dagger}, \\
p(a,t,0) =p(a,t,24),&(a,t)\in (0,a_\dagger)\times (0,T), \\
\partial_x p(a,t,0) =\partial_x  p(a,t,24),&(a,t)\in (0,a_\dagger)\times (0,T), \\
\displaystyle p(0,t,x)=\displaystyle \int_0^{a_\dagger} \beta(a)\displaystyle \int_{x-\eta}^{x+\eta}K(x,s)p(a,t,s)ds da,&(t,x)\in (0,T)\times (0,24), \\
\displaystyle p(a,0,x)=p_0(a,x) , &(a,x)\in (0,{a_\dagger} ) \times (0,24),
\end{array}
\right.
\end{equation}
where $Q_{{a_\dagger}}=(0,{a_\dagger} ) \times (0,T)\times (0,24)$ and
\[
Dp(a,t,x)={\lim_{\varepsilon \rightarrow 0} }\frac{p\left(
a+\varepsilon ,t+\varepsilon, x \right) -p\left( a,t,x\right) }{\varepsilon }
\]
is the directional derivative of $p$ with respect to  direction $\left(
1,1,0\right) .$ For $p$ smooth enough, it is easy to know that
\[
Dp=\frac {\partial p}{\partial t}+\frac {\partial p}{\partial a}.
\]
Here, $p(a,t,x)$ is the distribution of individuals of age $a\ge 0$ at time $t\ge 0$ and bitting at time $x \in [0,24] $,
$a_{\dagger }$ means the life expectancy of an individual and $T$ is a positive constant. As we announced, the mosquitoes can adapt their bitting time. Thus, we set their adapting model to be a $\Delta$ diffusion with a diffusive coefficient $\delta$. Moreover, $\beta(a)$ and $\mu(a)$ denote the natural fertility-rate and the natural death-rate of individual of age $a$, respectively. In fact, the new generation is also able to adapt the bitting time in order to maximize its fitness. Let $\eta$ be the maximum bitting time difference which the new generation can reach and we model the adaption of the new generation by a kernel $K$ as defined as below
\begin{equation*}\label{eq:02}
K(x,s)= \begin{cases}
   {(x-s)^{2}}{e^{-(x-s)^{2}}}, &s \in (0,24), \\
   0,  &\text{else}.
   \end{cases}
\end{equation*}
The control function  $u(a,t,x)$  represents the insecticidal effort, such as the use of ITNs and RIs.

In our paper, the main goal is to prove that there exists an optimal control $u$ in limited conditions, that is, $u$ is bounded by two functions $\varsigma_1$ and $\varsigma_2$ such that the insecticidal efficiency reaches the best. Since the control function $u$ is negative, it means that
we can deal with the following optimal problem
\[
(OH)\ \ \ \ \ \ \ \ \    Maximize \left\{-\displaystyle \int_{Q_{{a_\dagger}}}u(a,t,x)p^{u}(a,t,x)dtdxda\right\},
\]
subject to $u\in U$,
\[
U=\{u(a,t,x)\in L^{2}(Q_{{a_\dagger}})| \ \varsigma_{1}(a,t,x)\leq u(a,t,x)\leq \varsigma_{2}(a,t,x) \ a.e.\  in \ Q_{{a_\dagger}}\},
\]
where $\varsigma_{1},\varsigma_{2}\in L^{\infty}(Q_{{a_\dagger}}),\ \varsigma_{1}(a,t,x)\leq \varsigma_{2}(a,t,x) \leq 0 \ a.e.\  in \ Q_{{a_\dagger}}$ and
$p^{u}$ is the solution of system \eqref{eq:01}. Here, we say that the control $u^{*}\in U $ is optimal if
\[
\displaystyle \int_{Q_{{a_\dagger}}}u^{*}(a,t,x)p^{u^{*}}(a,t,x)dtdxda\leq \displaystyle \int_{Q_{{a_\dagger}}}u(a,t,x)p^{u}(a,t,x)dtdxda,
\]
for any $u\in U$. The pair $(u^{*},p^{u^{*}})$ is an optimal pair and $\displaystyle \int_{Q_{{a_\dagger}}}u^{*}p^{u^{*}}dtdxda$ is the optimal value of the cost functional.

Let us recall some history about the optimal control researches. Since 1985 when Brokate \cite{Bm} first proposed the optimal control of the population dynamical system with an age structure, it has been widely concerned and extensively studied by more and more researchers in the past few years. It is worth mentioning that the researches of Gurtin and Murphy \cite{Gm1,Gm2} about the optimal harvesting of age-structured populations provide an important basis for subsequent researches of the optimal control problem. As is well known, the optimal harvesting problem governed by nonlinear age dependent population dynamics with diffusion was considered by Ani\c{t}a \cite{As}, where he mainly discussed the impact of the control in homogeneous Neuman boundary conditions. For more rich results about the optimal control of an age structure with non-periodic boundary conditions, one can refer to \cite{As2,As3,Fk,Zc} and references cited therein. Note that the above results are about nonperiodic boundary conditions.

However, we have seen from the practical significance of biology that it is advantageous to consider age-structured models with periodic boundary conditions and nonlocal birth processes. We would like to refer to \cite{Al,Pe} for some studies about the optimal control problem with periodic boundary conditions. We also refer to \cite{Am,Kd,Lf,Lg} as reviewing references of the optimal control problem. Let us now mention some of our work about other aspects of system \eqref{eq:01} with periodic boundary conditions and nonlocal birth processes. In \cite{Ll1}, large time behaviour of the solution for such age-structured population model was considered. Moreover, we considered the local exact controllability of such age-structured problem in \cite{Ll2}. In this work, we study the optimal control of system \eqref{eq:01}.

From the biological point of view (one can refer to \cite{Ga,Gu,We}), we make the following hypotheses throughout this paper:
\begin{description}
\item[(J1)]$\mu(a)\in L^{\infty}_{loc}((0,a_{\dagger}))$, $\displaystyle \int_0^{a_{\dagger}} {\mu}(a)d a=+\infty $ and $\mu(a)\geq 0 \ a.e. \ in \ (0,a_{\dagger})$;
\item[(J2)]$\beta(a) \in L^{\infty}((0,a_{\dagger}))$, $\beta(a)\geq0 \ a.e. \ in \ (0,a_{\dagger})$;
\item[(J3)]$p_{0}(a,x)\in L^{2}((0,{a_\dagger} ) \times (0,24))$, $p_{0}(a,x)\geq 0  \ a.e. \ in \ (0,{a_\dagger} ) \times (0,24)$.
\end{description}

Now we state our main results.
\begin{theorem}\label{theorem1}
For any  $u\in U$, there exists a unique solution $p^{u}(a,t,x)\in L^2(Q_{a_\dagger})$ of the system \eqref{eq:01}.
\end{theorem}
\begin{theorem}\label{theorem2}
Problem $(OH)$ admits at least one optimal pair $(u^*,p^*)$.
\end{theorem}
\begin{theorem}\label{theorem3}
Let $(u^{*}(a,t,x),p^{*}(a,t,x))$ be an optimal pair for $(OH)$ and $q(a,t,x)$ be the solution of the following system
\begin{equation}\label{eq:000009}
\left\{ \begin{array}{lll}
Dq+ \delta \Delta q-\mu (a)q+ \beta(a)\displaystyle \int_{x-\eta}^{x+\eta}K(x,s)q(0,t,s)ds =- u^{*}q-u^{*},(a,t,x)\in Q_{a_\dagger}, \\
q(a,t,0) =q(a,t,24),\hspace{4.55cm}(a,t)\in (0,a_\dagger)\times (0,T), \\
\partial_x q(a,t,0) =\partial_x q(a,t,24),\hspace{3.8cm}(a,t)\in (0,a_\dagger)\times (0,T), \\
q(a_\dagger,t,x)=0,\hspace{5.8cm}(t,x)\in (0,T)\times (0,24), \\
q(a,T,x)=0 , \hspace{5.85cm}(a,x)\in (0,{a_\dagger} ) \times (0,24).
\end{array}
\right.
\end{equation}
Then, one has
\begin{equation*}
u^{*}(a,t,x)=\left\{
\begin{aligned}
&\varsigma_{1}(a,t,x),\ \ \text{if}\ \ q(a,t,x)>-1,\\
&\varsigma_{2}(a,t,x),\ \ \text{if}\ \ q(a,t,x)<-1.
\end{aligned}
\right.
\end{equation*}
\end{theorem}

This paper is organized as follows. In Section 2, we prove the existence of solutions and the comparison result for a linear model which is \eqref{eq:01} in general settings.
Section 3 is devoted to the proof of the existence of an optimal control of system \eqref{eq:01} by Mazur's Theorem.
Section 4 focuses on the necessary optimality conditions.

\section{Preliminaries}
\noindent
In this section, we study some properties of the following system, which is \eqref{eq:01} in general settings,
\begin{equation}\label{eq:2.1}
\left\{ \begin{array}{lll}
Dp- \delta \Delta p+\mu (a,t,x)p = f(a,t,x),&(a,t,x)\in Q_{a_\dagger}, \\
p(a,t,0) =p(a,t,24),&(a,t)\in (0,a_\dagger)\times (0,T), \\
\partial_x p(a,t,0) =\partial_x  p(a,t,24),&(a,t)\in (0,a_\dagger)\times (0,T), \\
\displaystyle p(0,t,x)=\displaystyle \int_0^{a_\dagger} \beta(a)\displaystyle \int_{x-\eta}^{x+\eta}K(x,s)p(a,t,s)ds da,&(t,x)\in (0,T)\times (0,24), \\
\displaystyle p(a,0,x)=p_0(a,x) , &(a,x)\in (0,{a_\dagger} ) \times (0,24),
\end{array}
\right.
\end{equation}
where $\beta$, $p_0$ are under the assumptions $(J2)$, $(J3)$, $\mu$ and $f$ satisfy
\begin{equation}\label{mu}
\mu(a,t,x)\in L^{\infty}_{loc}([0,a_{\dagger})\times [0,T]\times [0,24]),\ \mu(a,t,x)\ge 0 \text{ a.e. in $Q_{a_{\dagger}}$},
\end{equation}
\begin{equation*}
 f(a,t,x)\in L^{2}(Q_{a_{\dagger}}), \ \  f(a,t,x)\ge 0 \ \text{a.e. in $Q_{a_{\dagger}}$}.
\end{equation*}
Especially, we prove that there exists a unique solution of system \eqref{eq:2.1} and the comparison principle for system \eqref{eq:2.1}.

Before going further, we need an auxiliary lemma, which can be proved by following the proof of \cite[Lemma A2.7]{As}.

\begin{lemma}\label{lemma2.1}
For any $y_0(x)\in L^2(0,24)$, $g(t,x)\in L^2((0,T)\times(0,24))$, there exists a unique solution $y(t,x)\in L^2((0,T);H^1(0,24))\cup L^2_{loc}((0,T); H^2(0,24))$ of the following system
\begin{equation*}
\left\{ \begin{array}{lll}
\frac{\partial y}{\partial t}(t,x)- \delta\Delta y(t,x) = g(t,x),&(t,x)\in (0,T)\times (0,24), \\
y(t,0) =y(t,24),&t\in (0,T), \\
y'(t,0) =y'(t,24),&t\in (0,T), \\
y(0,x)=y_0(x) , &x\in (0,24).
\end{array}
\right.
\end{equation*}
\end{lemma}

\begin{remark}
It is known that there exists an orthogonal basis $\{\varphi_j\}_{j\in\mathbb{N}} \subset L^2(0,24)$ and $\{\lambda_j\}\subset \mathbb{R}^+$, $\lambda_0=0$, $\lambda_j\rightarrow +\infty$ as $j\rightarrow +\infty$ such that
\begin{equation*}
\left\{ \begin{array}{lll}
-\Delta \varphi_j(x)=\lambda_j \varphi_j(x), \text{ in $(0,24)$},\\
\varphi_j(0)=\varphi(24),\\
\varphi'_j(0)=\varphi'_j(24).
\end{array}
\right.
\end{equation*}
We can replace the basis in the proof of \cite[Lemma A2.7]{As} by our $\{\varphi_j\}_{j\in\mathbb{N}}$ and follow the same proof to get Lemma \ref{lemma2.1}.
\end{remark}

Let us first deal with the case when $\mu$ satisfies
\begin{description}
\item[(A)] $\mu\in L^{\infty}(Q_{a_{\dagger}})$, $\mu(a,t,x)\ge 0$ a.e. in $Q_{a_{\dagger}}$.
\end{description}

\begin{lemma}\label{lemma1}
For any fixed  $f(a,t,x)\in L^{2}(Q_{{a_\dagger}})$, $b(t,x)\in L^{2}((0,T)\times(0,24))$, there exists a unique solution $p_b(a,t,x)\in L^{2}(Q_{{a_\dagger}})$ of the following system
\begin{equation}\label{eq:111222401}
\left\{ \begin{array}{lll}
Dp- \delta \Delta p+\mu(a,t,x) p = f(a,t,x),&(a,t,x)\in Q_{a_\dagger}, \\
p(a,t,0) =p(a,t,24),&(a,t)\in (0,a_\dagger)\times (0,T), \\
\partial_x p(a,t,0) =\partial_x  p(a,t,24),&(a,t)\in (0,a_\dagger)\times (0,T), \\
\displaystyle p(0,t,x)=b(t,x),&(t,x)\in (0,T)\times (0,24), \\
\displaystyle p(a,0,x)=p_0(a,x) , &(a,x)\in (0,{a_\dagger} ) \times (0,24),
\end{array}
\right.
\end{equation}
where $\mu$ is under $(A)$.
\end{lemma}
\begin{proof}
Fix any $q(a,t,x)\in L^{2}(Q_{{a_\dagger}})$, we first prove that the following system has a unique solution $p_{b,q}(a,t,x)$,
\begin{equation}\label{eq:1112224}
\left\{ \begin{array}{lll}
Dp- \delta \Delta p+\mu q = f,&(a,t,x)\in Q_{a_\dagger}, \\
p(a,t,0) =p(a,t,24),&(a,t)\in (0,a_\dagger)\times (0,T), \\
\partial_x p(a,t,0) =\partial_x  p(a,t,24),&(a,t)\in (0,a_\dagger)\times (0,T), \\
\displaystyle p(0,t,x)=b(t,x),&(t,x)\in (0,T)\times (0,24), \\
\displaystyle p(a,0,x)=p_0(a,x) , &(a,x)\in (0,{a_\dagger} ) \times (0,24).
\end{array}
\right.
\end{equation}
Let $S$ be an arbitrary characteristic line of equation
\[
S=\{(a_{0}+s,t_{0}+s);s\in(0,\alpha)\},
\]
where $(a_0,t_0)\in \{0\}\times(0,T)\cup (0,a_{\dagger})\times \{0\}$ and $(a_0+\alpha,t_0+\alpha) \in \{a_{\dagger}\}\times (0,T) \cup (0,a_{\dagger})\times \{T\}$ and define
\begin{equation}\label{eq:111}
\left\{ \begin{array}{lll}
\widetilde{p}(s,x) = p(a_{0}+s,t_{0}+s,x),&(s,x)\in (0,\alpha)\times(0,24), \\
\widetilde{q}(s,x) = q(a_{0}+s,t_{0}+s,x),&(s,x)\in (0,\alpha)\times(0,24), \\
\widetilde{f}(s,x) = f(a_{0}+s,t_{0}+s,x),&(s,x)\in (0,\alpha)\times(0,24), \\
\widetilde{\mu}(s,x) = \mu(a_{0}+s,t_{0}+s,x),&(s,x)\in (0,\alpha)\times(0,24).
\end{array}
\right.
\end{equation}
According to Lemma \ref{lemma2.1}, the following system admits a unique solution $\widetilde{p}\in L^2((0,\alpha);H^1(0,24))\cap {L^{2}}_{loc}((0,\alpha);H^{2}(0,24))$,
\begin{equation}\label{eq:222}
\begin{cases}
\frac {\partial \widetilde{p}}{\partial s}- \delta \Delta \widetilde{p} =\widetilde{f}-\widetilde{\mu}\widetilde{q},&(s,x)\in (0,\alpha)\times(0,24), \\
\partial _{x} \widetilde{p}(s,0)=\partial _{x} \widetilde{p}(s,24),&s\in (0,\alpha), \\
\widetilde{p}(s,0)= \widetilde{p}(s,24),&s\in (0,\alpha), \\
\widetilde{p}(0,x)=\left\{
\begin{aligned}
&b(t_{0},x),\ \ \ \  a_{0}=0,\ \  x \in (0,24),\\
&p_{0}(a_{0},x),\ \ \ t_{0}=0,\ \ x \in (0,24).
\end{aligned}
\right.
\end{cases}
\end{equation}
In fact, multiplying the first equation of system \eqref{eq:222} by $\widetilde{p}$ and integrating on $(0,s)\times(0,24)$, one has
\[\|\widetilde{p}(s)\|^2_{L^{2}(0,24)}\leq \|\widetilde{p}(0)\|^2_{L^{2}(0,24)}+\|\widetilde{f}-\widetilde{\mu}\widetilde{q}\|^2_{L^{2}((0,\alpha)\times(0,24))}+\displaystyle \int_0^{s}\|\widetilde{p}(\tau)\|^2_{L^{2}(0,24)}d\tau.\]
Then by a lemma from Bellman (see in Appendix) we get
\begin{equation}\label{eq:2.6}
\|\widetilde{p}(s)\|^2_{L^{2}(0,24)}\leq C(\|\widetilde{p}(0)\|^2_{L^{2}(0,24)}+\|\widetilde{f}-\widetilde{\mu}\widetilde{q}\|^2_{L^{2}((0,\alpha)(0,24))})e^{\alpha},\ \forall s \in[0,\alpha]
\end{equation}
Now let us denote
\[
p_{b,q}(a_{0}+s,t_{0}+s,x)=\widetilde{p}(s,x),\ (s,x)\in (0,\alpha)\times(0,24)
\]
for any characteristic line $S$. It follows from Lemma \ref{lemma2.1} and \eqref{eq:2.6} that $p_{b,q}\in  L^{2}(S;H^{1}(0,24))\cap {L^{2}}_{loc}(S;H^{2}(0,24))$
for almost any characteristic line $S$, and $p_{b,q}$ satisfies
\begin{equation}\label{eq:20202}
\left\{ \begin{array}{lll}
Dp_{b,q}- \delta \Delta p_{b,q}+\mu (a,t,x)q = f(a,t,x),&(a,t,x)\in Q_{a_\dagger}, \\
p_{b,q}(a,t,0) =p_{b,q}(a,t,24),&(a,t)\in (0,a_\dagger)\times (0,T), \\
\partial_x p_{b,q}(a,t,0) =\partial_x p_{b,q}(a,t,24),&(a,t)\in (0,a_\dagger)\times (0,T), \\
\displaystyle p_{b,q}(0,t,x)=b(t,x),&(t,x)\in (0,T)\times (0,24), \\
\displaystyle p_{b,q}(a,0,x)=p_0(a,x) , &(a,x)\in (0,{a_\dagger} ) \times (0,24).
\end{array}
\right.
\end{equation}

Now we show that $p_{b,q}(a,t,x)\in L^{2}(Q_{a_\dagger})$. It is known that there exists an orthonormal basis $\{\varphi_j\}_{j\in\mathbb{N}} \subset L^2(0,24)$ and $\{\lambda_j\}\subset \mathbb{R}^+$, $\lambda_0=0$, $\lambda_j\rightarrow +\infty$ as $j\rightarrow +\infty$ such that
\begin{equation*}
\left\{ \begin{array}{lll}
-\Delta \varphi_j(x)=\lambda_j \varphi_j(x), \text{ in $(0,24)$},\\
\varphi_j(0)=\varphi(24),\\
\varphi'_j(0)=\varphi'_j(24).
\end{array}
\right.
\end{equation*}
Then, one has that
$$f(a,t,x)-\mu(a,t,x) q=\sum_{j=1}^{\infty} v^j(a,t)\varphi_j(x),\text{in $L^2(0,24)$, a.e. $(a,t)\in(0,a_{\dagger})\times(0,T)$},$$
$$b(t,x)=\sum_{j=1}^{\infty} b^j(t)\varphi_j(x), \text{ in $L^2(0,24)$, a.e. $t\in(0,T)$},$$
$$p_0(a,x)=\sum_{j=1}^{\infty} p_0^j(a)\varphi_j(x), \text{ in $L^2(0,24)$, a.e. $a\in(0,a_{\dagger})$}.$$
Furthermore, $p_{b,q}(a,t,x)$ has the following expression
$$p_{b,q}(a,t,x):=\sum_{j=1}^{\infty} p_{b,q}^j(a,t)\varphi_j(x),\text{ in $L^2(0,24)$ a.e. $(a,t)\in(0,a_{\dagger})\times(0,T)$}.$$
By substituting it into \eqref{eq:20202}, one gets that $p_{b,q}^j(a,t)$ satisfies
\begin{equation*}
\left\{ \begin{array}{lll}
Dp_{b,q}^j+\lambda_j \delta p_{b,q}^j=v^j(a,t), &(a,t)\in (0,a_{\dagger})\times(0,T),\\
p_{b,q}^j(a,t)\varphi_j(0)=p_{b,q}^j(a,t)\varphi_j(24),&(a,t)\in (0,a_{\dagger})\times(0,T),\\
p_{b,q}^j(a,t)\varphi_j^{'}(0)=p_{b,q}^j(a,t)\varphi_j^{'}(24),&(a,t)\in (0,a_{\dagger})\times(0,T),\\
p_{b,q}^j(0,t)=b^j(t), &t\in(0,T)\\
p_{b,q}^j(a,0)=p_0^j(a), &a\in(0,a_{\dagger}).
\end{array}
\right.
\end{equation*}
One can follow the computation of Lemma 4.1 in Ani\c{t}a $\cite [p_{.} 113-114]{As}$ and get that $p_{b,q}(a,t,x)\in L^{2}(Q_{a_\dagger})$ satisfies
 \begin{align}\label{eq:0001167}
\|p_{b,q}\|^{2}_{L^{2}(Q_{{a_\dagger}})}\leq  e^{T}(\|p_{0}\|^{2}_{L^{2}((0,a_\dagger)\times (0,24))}+\|b\|^{2}_{L^{2}((0,T)\times (0,24))}+\|f-\mu q\|^{2}_{L^{2}(Q_{a_\dagger})}).
\end{align}

For an arbitrary $q(a,t,x)\in L^{2}(Q_{{a_\dagger}})$, we have obtained that system \eqref{eq:1112224} has a solution $p_{b,q}\in L^2(Q_{a_{\dagger}})$.
Let us set a mapping $\Pi: L^2(Q_{a_\dagger})\rightarrow L^2(Q_{a_\dagger})$ by $\Pi (q_{i}(a,t,x))={p_{b,q}}_{i}(a,t,x)$. Take any two functions $q_1$, $q_2\in L^2(Q_{a_\dagger})$ and then $p_{b,q_1}-p_{b,q_2}$ satisfies
\begin{equation}\label{eq:202022233}
\left\{ \begin{array}{lll}
D({p_{b,q}}_{1}-{p_{b,q}}_{2})- \delta \Delta({p_{b,q}}_{1}-{p_{b,q}}_{2})+\mu (q_{1}-q_{2}) = 0,&(a,t,x)\in Q_{a_\dagger}, \\
({p_{b,q}}_{1}-{p_{b,q}}_{2})(a,t,0) =({p_{b,q}}_{1}-{p_{b,q}}_{2})(a,t,24),&(a,t)\in (0,a_\dagger)\times (0,T), \\
\partial_x ({p_{b,q}}_{1}-{p_{b,q}}_{2})(a,t,0) =\partial_x ({p_{b,q}}_{1}-{p_{b,q}}_{2})(a,t,24),&(a,t)\in (0,a_\dagger)\times (0,T), \\
({p_{b,q}}_{1}-{p_{b,q}}_{2})(0,t,x)=0,&(t,x)\in (0,T)\times (0,24), \\
({p_{b,q}}_{1}-{p_{b,q}}_{2})(a,0,x)=0,&(a,x)\in (0,{a_\dagger} ) \times (0,24).
\end{array}
\right.
\end{equation}
By the result of \eqref{eq:0001167}, one has
\begin{align*}
\|{p_{b,q}}_{1}-{p_{b,q}}_{2}\|^{2}_{L^{2}(Q_{{a_\dagger}})}
\leq  e^{T}(\|\mu (q_{1}-q_{2})\|^{2}_{L^{2}(Q_{a_\dagger})}), \ in \ L^{2}(Q_{a_\dagger}).
\end{align*}
Obviously, when $T$ is small enough, $p_{b,q}(a,t,x)$ is a contraction mapping with respect to $q(a,t,x)$. Consequently, there exists a unique solution $p_{b}$ of system \eqref{eq:111222401} for sufficient small $T$. However, one can extend $T$ by following previous steps for $t\in(T,2T)$.
Thus, system \eqref{eq:111222401} has a unique solution $p_{b}\in L^{2}(Q_{a_\dagger})$.
\end{proof}

One can follow the same idea of the proof of \cite[Lemma 4.1.2]{As} to get the following Lemma.
\begin{lemma}\label{lemma2.2}
For any $b_1(t,x)$, $b_2(t,x)\in L^2((0,T)\times(0,24))$, $0\le b_1(t,x)\le b_2(t,x)$ a.e. in $(0,T)\times(0,24)$, one has
$$0\le p_{b_1}(a,t,x)\le p_{b_2}(a,t,x), \text{ a.e. in $Q_{a_{\dagger}}$},$$
where $p_{b_1}(a,t,x)$ and $p_{b_2}(a,t,x)$ are the solutions of system \eqref{eq:111222401} under $(A)$ with $b_1(a,t,x)$ and $b_2(a,t,x)$ respectively.
\end{lemma}

\begin{lemma}\label{lemma2}
There exists a unique solution $p(a,t,x)\in L^2(Q_{a_\dagger})$ of system \eqref{eq:2.1} under $(A)$.
\end{lemma}

\begin{proof}
Let us define an operator $\mathscr{F}: L^2((0,T)\times (0,24))\rightarrow L^2((0,T)\times (0,24))$ by
\[
(\mathscr{F}b)(t,x)=\displaystyle \int_0^{a_\dagger} \beta(a)\displaystyle \int_{x-\eta}^{x+\eta}K(x,s)p_{b}(a,t,s)ds da,\ \text{ a.e. in } (0,T)\times (0,24).
\]
For any fixed $b_{i}\in L^2((0,T)\times (0,24))$ $(i=1,2)$, let $p_{b_{1}},p_{b_{2}}\in L^2(Q_{a_{\dagger}})$ be the solutions of system \eqref{eq:2.1} with $b_1(a,t,x)$, $b_2(a,t,x)$ respectively. Let $v(a,t,x)=p_{b_{1}}(a,t,x)-p_{b_{2}}(a,t,x)$ and then $v(a,t,x)$ satisfies
\begin{equation}\label{eq:102345}
\left\{ \begin{array}{lll}
Dv- \delta \Delta v+\mu (a,t,x)v = 0,&(a,t,x)\in Q_{a_\dagger}, \\
v(a,t,0) =v(a,t,24),&(a,t)\in (0,a_\dagger)\times (0,T), \\
\partial_x v(a,t,0) =\partial_x  v(a,t,24),&(a,t)\in (0,a_\dagger)\times (0,T), \\
v(0,t,x)=b_{1}(t,x)-b_{2}(t,x),&(t,x)\in (0,T)\times (0,24), \\
v(a,0,x)=0 , &(a,x)\in (0,{a_\dagger} ) \times (0,24).
\end{array}
\right.
\end{equation}
Then it follows by the computation of Lemma 4.1 in Ani\c{t}a $\cite [p_{.} 116]{As}$ that
\begin{align*}
\int_0^T e^{-\lambda t} \|v(t)\|^2_{L^2((0,a_{\dagger})\times(0,24))} dt \le \frac{1}{\lambda} \int_0^T e^{-\lambda t} \|b_1(t)-b_2(t)\|^2_{L^2(0,24)} dt
\end{align*}
for any $\lambda>0$. Consider $L^2((0,T)\times (0,24))$ with the norm
$$\|b\|=\left(\int_0^T e^{-\lambda t} \|b(t)\|^2_{L^2(0,24)} dt \right)^2, \text{ for any $b\in L^2((0,T)\times(0,24))$}.$$
Then one has
\begin{align*}
&\|\mathscr{F}b_{1}-\mathscr{F}b_{2}\|^{2}\\
=&\displaystyle \int_0^{T} e^{-\lambda t}\|\displaystyle \int_0^{a_\dagger} \beta(a)\displaystyle \int_{x-\eta}^{x+\eta}K(x,s)(p_{b_{1}}(a,t,s)-p_{b_{2}}(a,t,s))ds da\|^{2}_{L^{2}(0,24)}dt\\
\leq & C \displaystyle \int_0^{a_\dagger} \beta^{2}(a)da \displaystyle \int_0^{T} e^{-\lambda t}\|v(t)\|^{2}_{L^2((0,a_\dagger)\times (0,24))}dt\\
\leq & \frac{C}{\lambda} \displaystyle \int_0^{a_\dagger} \beta^{2}(a)da\|b_1-b_2\|^{2},
\end{align*}
where $C$ is an appropriate positive constant related to $K(x,s)$.
One can choose $\lambda$ large such that $\lambda >C \displaystyle \int_0^{a_\dagger} \beta^{2}(a)da $ and then $\mathscr{F}$ is a contraction mapping on $L^2((0,T)\times (0,24))$ with the norm $\|\cdot\|$. This completes the proof.
\end{proof}

From Lemma \ref{lemma2}, one gets that the operator $\mathscr{F}$ is a contraction mapping. Moreover, combined with Lemma \ref{lemma2.2}, one can follow the rest of the proof of \cite[Lemma 4.1.1]{As} to get the following comparison principle for \eqref{eq:2.1}.

\begin{lemma}\label{lemma3}
If $p_{i}(i\in {1,2})$ are the solutions of the following systems
\begin{equation*}
\left\{ \begin{array}{lll}
Dp_{i}- \delta \Delta p_{i}+\mu_{i} (a,t,x)p_{i} = f_{i},&(a,t,x)\in Q_{a_\dagger}, \\
p_{i}(a,t,0) =p_{i}(a,t,24),&(a,t)\in (0,a_\dagger)\times (0,T), \\
\partial_x p_{i}(a,t,0) =\partial_x  p_{i}(a,t,24),&(a,t)\in (0,a_\dagger)\times (0,T), \\
\displaystyle p_{i}(0,t,x)=\displaystyle \int_0^{a_\dagger} \beta_{i}(a)\displaystyle \int_{x-\eta}^{x+\eta}K(x,s)p_{i}(a,t,s)ds da,&(t,x)\in (0,T)\times (0,24), \\
\displaystyle p_{i}(a,0,x)=p_{0i}(a,x) , &(a,x)\in (0,{a_\dagger} ) \times (0,24),
\end{array}
\right.
\end{equation*}
where $\mu_{1}\geq\mu_{2}$, $f_{1}\leq f_{2}$, $\beta_{1}\leq\beta_{2}$, $p_{01}\leq  p_{02}$ and $\mu_1$, $\mu_2$ satisfy $(A)$,
then
\[
0\leq p_{1}(a,t,x)\leq p_{2}(a,t,x) \ \ a.e. \  in \  Q_{a_\dagger}.
\]
\end{lemma}

By referring to the proof of \cite[Theorem 4.1.3, Theorem 4.1.4]{As} for the case when $\mu(a,t,x)$ satisfies \eqref{mu}, one can define
$$\mu^N(a,t,x)=\min\{\mu(a,t,x),N\}, \text{ for any $N\in\mathbb{N}^+$},$$
and denote $p_N(a,t,x)$ to be the solution of system \eqref{eq:2.1} with $\mu_N$. Passing to the limit as $N\rightarrow +\infty$ for $p_N(a,t,x)$, one can get the solution of system \eqref{eq:2.1}. Then by the results of Lemma \ref{lemma2} and Lemma \ref{lemma3}, we can get the following lemma.

\begin{lemma}\label{lemma2.7}
There is a unique solution $p(a,t,x)\in L^2(Q_{a_{\dagger}})$ of system \eqref{eq:2.1} with $\mu$ satisfying \eqref{mu}. If $p_{i}(i\in {1,2})$ are the solutions of system \eqref{eq:2.1} with $\mu_1$, $f_1$, $\beta_1$, $p_{01}$ and $\mu_2$, $f_2$, $\beta_2$, $p_{02}$ respectively ($\mu_1$, $\mu_2$ satisfy \eqref{mu}) and $\mu_{1}\geq\mu_{2}$, $f_{1}\leq f_{2}$, $\beta_{1}\leq\beta_{2}$, $p_{01}\leq  p_{02}$, then
\[
0\leq p_{1}(a,t,x)\leq p_{2}(a,t,x) \ \ a.e. \  in \  Q_{a_\dagger}.
\]
\end{lemma}

According to Lemma \ref{lemma2.7} , we obtain the result of Theorem \ref{theorem1} directly.

\section{Existence of an optimal control}
\noindent
In this section, our main job is to obtain the existence of an optimal control of \eqref{eq:01} by  Mazur's Theorem, that is, we prove Theorem \ref{theorem2}.
\vskip 0.3cm

\begin{proof}[Proof of Theorem \ref{theorem2}]
Let $\Psi: U\rightarrow \mathbf{R^{+}}$ be defined by
\[
\Psi(u)=\displaystyle \int_{Q_{{a_{\dagger}}}}u(a,t,x)p^{u}(a,t,x)dtdxda.
\]
Then by the definition of $u(a,t,x)$, we have
\[
\displaystyle \int_{Q_{{a_{\dagger}}}}\varsigma_{1}(a,t,x)\overline{p}(a,t,x)dtdxda\leq\Psi(u)\leq 0,
\]
where $\overline{p}(a,t,x)$ is a solution of system \eqref{eq:01} with $u\equiv0$, $\mu\equiv0$, $\beta=\|\beta(a)\|_{L^{\infty}(0,a_\dagger)}$, $p_{0}=\|p_{0}\|_{L^{\infty}((0,a_\dagger)\times (0,24))}$, that is,
\begin{equation}\label{eq:333}
\left\{ \begin{array}{lll}
Dp- \delta \Delta p = 0,&(a,t,x)\in Q_{a_\dagger}, \\
p(a,t,0) =p(a,t,24),&(a,t)\in (0,a_\dagger)\times (0,T), \\
\partial_x p(a,t,0) =\partial_x  p(a,t,24),&(a,t)\in (0,a_\dagger)\times (0,T), \\
\displaystyle p(0,t,x)=\displaystyle \int_0^{a_\dagger} \beta(a)\displaystyle \int_{x-\eta}^{x+\eta}K(x,s)p(a,t,s)ds da,&(t,x)\in (0,T)\times (0,24), \\
\displaystyle p(a,0,x)=p_0 , &(a,x)\in (0,{a_\dagger} ) \times (0,24).
\end{array}
\right.
\end{equation}
Thus, we can assume that
\[
d=\inf_{u\in U}\Psi(u),
\]
and there exists a sequence $\{u_{N}\} \in U, N  \in N^{*}$ such that
\[
d\leq \Psi(u_{N})< d+\frac{1}{N},
\]
\begin{equation}\label{eq:555}
\Psi(u_{N})\rightarrow d.
\end{equation}
Since the result of Lemma \ref{lemma2.7}, one obtains
\[
0\leq p^{u_{N}}(a,t,x)\leq \overline{p}(a,t,x) \ \ \ a.e. \ in  \ Q_{a_\dagger}.
\]
Thus there exists a subsequence which still be denoted by $\{u_{N}\}$ such that
\[
p^{u_{N}}\rightarrow p^{*} \ \ \ weakly\  in \ L^{2}(Q_{{a_\dagger}}).
\]
By Mazur's Theorem, one has that
$\forall \epsilon >0$, there exists $\lambda_{i}^{N}\geq0$, $\sum_{i=N+1}^{k_{N}} {\lambda_{i}^{N}}=1$ such that
\[
\|p^{*}-\sum_{i=N+1}^{k_{N}} {\lambda_{i}^{N}}p^{u_{i}}\|_{L^{2}(Q_{{a_\dagger}})}\leq \epsilon,
\]
where $k_{N}\geq N+1$.
Now we denote
\[
\widetilde{p}_{N}(a,t,x)=\sum_{i=N+1}^{k_{N}} {\lambda_{i}^{N}}p^{u_{i}}(a,t,x),
\]
therefore,
\[
\widetilde{p}_{N}\rightarrow p^{*}   \ \ \ in  \ L^{2}(Q_{{a_\dagger}}).
\]

Now we consider the sequence $\{\widetilde{u}_{N} \}$ of controls $\{{u}_{i} \}$. Here $\widetilde{u}_{N} (a,t,x)$ is defined by
\begin{equation*}
\widetilde{u}_{N} (a,t,x)=\left\{
\begin{aligned}
&\frac{\sum_{i=N+1}^{k_{N}} {\lambda_{i}^{N}}{u}_{i}(a,t,x)p^{u_{i}}(a,t,x)}{\sum_{i=N+1}^{k_{N}} {\lambda_{i}^{N}}p^{u_{i}}}(a,t,x),\ \ if \sum_{i=N+1}^{k_{N}} {\lambda_{i}^{N}}p^{u_{i}}\neq 0,\\
&\varsigma_{1}(a,t,x),\ \ \ \ \ \ \ \ \ \ \ \ \ \ \ \ \ \ \  \ \ \ \ \ \ \ \ \ \ \ \ \ \ \ \ \ \ \ \ \ \ \ \ \ \ if \sum_{i=N+1}^{k_{N}} {\lambda_{i}^{N}}p^{u_{i}}= 0.
\end{aligned}
\right.
\end{equation*}
It is easy to check that $\widetilde{u}_{N} \in U$. Thus, one learns that there exists a subsequence $\{\widetilde{u}_{N} \}_{N\in N^{*}}$ such that
\[
\widetilde{u}_{N}\rightarrow u^{*} \ \ \ weakly\ \ in \ \ L^{2}(Q_{{a_\dagger}}).
\]
Obviously, $\widetilde{p}_{N}(a,t,x)$ is a solution of
\begin{equation}\label{eq:444}
\left\{ \begin{array}{lll}
Dp- \delta \Delta p+\mu (a)p = \widetilde{u}_{N}(a,t,x)p,&(a,t,x)\in Q_{a_\dagger}, \\
p(a,t,0) =p(a,t,24),&(a,t)\in (0,a_\dagger)\times (0,T), \\
\partial_x p(a,t,0) =\partial_x  p(a,t,24),&(a,t)\in (0,a_\dagger)\times (0,T), \\
\displaystyle p(0,t,x)=\displaystyle \int_0^{a_\dagger} \beta(a)\displaystyle \int_{x-\eta}^{x+\eta}K(x,s)p(a,t,s)ds da,&(t,x)\in (0,T)\times (0,24), \\
\displaystyle p(a,0,x)=p_0(a,x) , &(a,x)\in (0,{a_\dagger} ) \times (0,24).
\end{array}
\right.
\end{equation}
Passing to the limit in \eqref{eq:444}, we get
\begin{equation*}
\left\{ \begin{array}{lll}
Dp^{*}- \delta \Delta p^{*}+\mu (a)p^{*} = u^{*} p^{*},&(a,t,x)\in Q_{a_\dagger}, \\
p^{*}(a,t,0) =p^{*}(a,t,24),&(a,t)\in (0,a_\dagger)\times (0,T), \\
\partial_x p^{*}(a,t,0) =\partial_x  p^{*}(a,t,24),&(a,t)\in (0,a_\dagger)\times (0,T), \\
\displaystyle p^{*}(0,t,x)=\displaystyle \int_0^{a_\dagger} \beta(a)\displaystyle \int_{x-\eta}^{x+\eta}K(x,s)p^{*}(a,t,s)ds da,&(t,x)\in (0,T)\times (0,24), \\
\displaystyle p^{*}(a,0,x)=p_0(a,x) , &(a,x)\in (0,{a_\dagger} ) \times (0,24).
          \end{array}
  \right. \end{equation*}
It means that $p^{*}$ is the solution of system \eqref{eq:01} corresponding to $u^{*}$.

Therefore, we have
\begin{align*}
\sum_{i=N+1}^{k_{N}} {\lambda_{i}^{N}}\Psi({u_{i}})=&\sum_{i=N+1}^{k_{N}} {\lambda_{i}^{N}}\displaystyle \int_{Q_{{a_{\dagger}}}}{u}_{i}(a,t,x)p^{u_{i}}(a,t,x)dx dt da \\
=&\displaystyle \int_{Q_{{a_{\dagger}}}}\widetilde{u}_{N}(a,t,x)\widetilde{p}_{N}(a,t,x)dx dt da \\
\rightarrow &\displaystyle \int_{Q_{{a_{\dagger}}}}{u}^{*}(a,t,x){p}^{*}(a,t,x)dx dt da \\
= &\Psi(u^{*}).
\end{align*}
Using \eqref{eq:555} and the last equation, we can conclude that $d=\Psi(u^{*})$.
\end{proof}

\section{Necessary optimality conditions}
\noindent
In this section, our goal is to obtain the necessary optimality conditions of $(OH)$ which is Theorem \ref{theorem3}.

\begin{proof}[Proof of Theorem \ref{theorem3}]
First of all, we can get that system \eqref{eq:000009} has a unique solution $q(a,t,x)\in L^2(Q_{a_\dagger})$ by the same method as in the proof of the existence and uniqueness of solutions of system \eqref{eq:01} in Section 2.

Since $(u^{*},p^{*})$ is an optimal pair for $(OH)$, we have
\[
\displaystyle \int_{Q_{{a_{\dagger}}}}u^{*}p^{u^{*}}dtdxda\leq \displaystyle \int_{Q_{{a_{\dagger}}}}(u^{*}+\epsilon v)p^{u^{*}+\epsilon v}dtdxda
\]
for any $\epsilon>0$ small enough, arbitrary $v(a,t,x)\in L^{\infty}(Q_{a_{\dag}})$ such that
\begin{equation*}
\left\{ \begin{array}{lll}
v(a,t,x)\leq 0,\ \ \text{if}\ \  u^{*}(a,t,x)=\varsigma_{2}(a,t,x),\\
v(a,t,x)\geq 0,\ \ \text{if}\ \ u^{*}(a,t,x)=\varsigma_{1}(a,t,x),
\end{array}
\right.
\end{equation*}
which implies
\begin{align}\label{eq:00023}
\displaystyle \int_{Q_{{a_{\dagger}}}}u^{*}\frac{p^{u^{*}+\epsilon v}-p^{u^{*}}}{\epsilon}dtdxda+\displaystyle \int_{Q_{{a_{\dagger}}}} vp^{u^{*}+\epsilon v}dtdxda\geq0.
\end{align}
Let $z^{\epsilon}(a,t,x)=\frac{p^{u^{*}+\epsilon v}(a,t,x)-p^{u^{*}}(a,t,x)}{\epsilon}$, $y^{\epsilon}(a,t,x)=\epsilon z^{\epsilon}(a,t,x)$, then $y^{\epsilon}(a,t,x)$ satisfies
\begin{equation*}
\left\{ \begin{array}{lll}
Dy^{\epsilon}- \delta \Delta y^{\epsilon}+\mu (a)y^{\epsilon} = u^{*}y^{\epsilon}+\epsilon v p^{u^{*}+ \epsilon v},&(a,t,x)\in Q_{a_\dagger}, \\
y^{\epsilon}(a,t,0) =y^{\epsilon}(a,t,24),&(a,t)\in (0,a_\dagger)\times (0,T), \\
\partial_x y^{\epsilon}(a,t,0) =\partial_x  y^{\epsilon}(a,t,24),&(a,t)\in (0,a_\dagger)\times (0,T), \\
\displaystyle y^{\epsilon}(0,t,x)=\displaystyle \int_0^{a_\dagger} \beta(a)\displaystyle \int_{x-\eta}^{x+\eta}K(x,s)y^{\epsilon}(a,t,s)ds da,&(t,x)\in (0,T)\times (0,24), \\
\displaystyle y^{\epsilon}(a,0,x)=0 , &(a,x)\in (0,{a_\dagger} ) \times (0,24).
\end{array}
\right.
\end{equation*}
Multiplying the first equation by $y^{\epsilon}$ and integrating on $Q_{t}=(0,a_\dagger)\times (0,t)\times (0,24)$, one obtains
\begin{align*}
\|y^{\epsilon}(t)\|_{L^{2}((0,a_\dagger)\times (0,24))}^{2}
\leq C \displaystyle \int_0^{t} \|y^{\epsilon}(s)\|_{L^{2}((0,a_\dagger)\times (0,24))}^{2}ds+ \epsilon \displaystyle \int_{Q_t} |v| p^{u^{*}+ \epsilon v}|y^{\epsilon}|dsdxda.
\end{align*}
Then by the result of Lemma \ref{lemma2.7} and Bellman's Lemma (see in Appendix), we get
\begin{align*}
\|y^{\epsilon}(t)\|_{L^{2}((0,a_\dagger)\times (0,24))}^{2}
&\leq \epsilon^{2} \displaystyle \int_{Q_{{a_{\dagger}}}} |v|^{2} \overline{p}^{2}dtdxda+(1+C) \displaystyle \int_0^{t} \|y^{\epsilon}(s)\|_{L^{2}((0,a_\dagger)\times (0,24))}^{2}ds\\
&\leq \epsilon^{2} e^{(1+C)t} \displaystyle \int_{Q_{{a_{\dagger}}}} |v|^{2} \overline{p}^{2}dtdxda
\end{align*}
where $\overline{p}(a,t,x)$ is a solution of system \eqref{eq:333}, $t\in[0,T]$ and $C$ is a positive constant.
This implies that
\begin{align}\label{eq:00011}
y^{\epsilon}\rightarrow 0 \  \ in \ \ L^{\infty}(0,T; L^{2}((0,a_\dagger)\times (0,24)) )\ \ as \ \ \epsilon\rightarrow 0^{+}.
\end{align}
So the following convergence holds
\[
p^{u^{*}+\epsilon v}\rightarrow p^{u^{*}}\  \ in \ \ L^{\infty}(0,T; L^{2}((0,a_\dagger)\times (0,24))) \ \ as \ \ \epsilon\rightarrow 0^{+}.
\]

Recalling the definition of $z^{\epsilon}(a,t,x)$, one has that $z^{\epsilon}(a,t,x)$ satisfies
\begin{equation*}
\left\{ \begin{array}{lll}
Dz^{\epsilon}- \delta \Delta z^{\epsilon}+\mu (a)z^{\epsilon} = u^{*}z^{\epsilon}+v p^{u^{*}+ \epsilon v},&(a,t,x)\in Q_{a_\dagger}, \\
z^{\epsilon}(a,t,0) =z^{\epsilon}(a,t,24),&(a,t)\in (0,a_\dagger)\times (0,T), \\
\partial_x z^{\epsilon}(a,t,0) =\partial_x z^{\epsilon}(a,t,24),&(a,t)\in (0,a_\dagger)\times (0,T), \\
\displaystyle z^{\epsilon}(0,t,x)=\displaystyle \int_0^{a_\dagger} \beta(a)\displaystyle \int_{x-\eta}^{x+\eta}K(x,s)z^{\epsilon}(a,t,s)ds da,&(t,x)\in (0,T)\times (0,24), \\
\displaystyle z^{\epsilon}(a,0,x)=0 , &(a,x)\in (0,{a_\dagger} ) \times (0,24).
\end{array}
\right.
\end{equation*}
Let $h^{\epsilon}(a,t,x)=z^{\epsilon}(a,t,x)-z(a,t,x)$, where $z(a,t,x)$ is a solution of the following system
\begin{equation*}
\left\{ \begin{array}{lll}
Dz- \delta \Delta z+\mu (a)z = u^{*}z+v p^{u^{*}},&(a,t,x)\in Q_{a_\dagger}, \\
z(a,t,0) =z(a,t,24),&(a,t)\in (0,a_\dagger)\times (0,T), \\
\partial_x z(a,t,0) =\partial_x z(a,t,24),&(a,t)\in (0,a_\dagger)\times (0,T), \\
\displaystyle z(0,t,x)=\displaystyle \int_0^{a_\dagger} \beta(a)\displaystyle \int_{x-\eta}^{x+\eta}K(x,s)z(a,t,s)ds da,&(t,x)\in (0,T)\times (0,24), \\
\displaystyle z(a,0,x)=0 , &(a,x)\in (0,{a_\dagger}) \times (0,24).
\end{array}
\right.
\end{equation*}
Following the above proof step by step, we can get that
\begin{align*}
\|h^{\epsilon}(t)\|_{L^{2}((0,a_\dagger)\times (0,24))}^{2}\leq e^{(1+C)t} \displaystyle \int_{Q_{{a_{\dagger}}}} |v|^{2} |y^{\epsilon}|^{2}dtdxda.
\end{align*}
Using \eqref{eq:00011}, one obtains
\[
z^{\epsilon}\rightarrow z\  \ in \ \ L^{\infty}(0,T; L^{2}((0,a_\dagger)\times (0,24)))\ \ as \ \ \epsilon\rightarrow 0^{+}.
\]
Passing to the limit in \eqref{eq:00023}, it follows
\begin{align}\label{eq:00024}
\displaystyle \int_{Q_{{a_{\dagger}}}}u^{*}zdtdxda+\displaystyle \int_{Q_{{a_{\dagger}}}} vp^{u^{*}}dtdxda\geq0,
\end{align}
for arbitrary $v(a,t,x)\in L^{\infty}(Q_{a_{\dag}})$ such that
\begin{equation*}
\left\{ \begin{array}{lll}
v(a,t,x)\leq 0,\ \ \text{if}\ \  u^{*}(a,t,x)=\varsigma_{2}(a,t,x),\\
v(a,t,x)\geq 0,\ \ \text{if}\ \ u^{*}(a,t,x)=\varsigma_{1}(a,t,x).
\end{array}
\right.
\end{equation*}
Multiplying the first equation of system \eqref{eq:000009} by $z(a,t,x)$ and integrating on $Q_{a_{\dagger}}$, we get
\begin{align}\label{eq:00025}
\displaystyle \int_{Q_{{a_{\dagger}}}}vp^{u^{*}}qdtdxda=\displaystyle \displaystyle \int_{Q_{{a_{\dagger}}}}u^{*}zdtdxda.
\end{align}
Combining \eqref{eq:00024} with \eqref{eq:00025}, we learn that
\begin{align}
\displaystyle \int_{Q_{{a_{\dagger}}}}vp^{u^{*}}(q+1)dtdxda\geq 0,
\end{align}
for arbitrary $v(a,t,x)\in L^{\infty}(Q_{a_{\dag}})$ such that
\begin{equation*}
\left\{ \begin{array}{lll}
v(a,t,x)\leq 0,\ \ \text{if} \ \ u^{*}(a,t,x)=\varsigma_{2}(a,t,x),\\
v(a,t,x)\geq 0,\ \ \text{if} \ \ u^{*}(a,t,x)=\varsigma_{1}(a,t,x).
\end{array}
\right.
\end{equation*}

For any $(a,t,x)\in Q_{a_{\dag}}$, if $p^{u^{*}}(a,t,x)\neq 0$ holds, we can conclude that
\begin{equation*}
u^{*}(a,t,x)=\left\{
\begin{aligned}
&\varsigma_{1}(a,t,x),\ \ \text{if}\ \ q(a,t,x)>-1,\\
&\varsigma_{2}(a,t,x),\ \ \text{if}\ \ q(a,t,x)<-1.
\end{aligned}
\right.
\end{equation*}

We now consider the set $B=\{(a,t,x)\in Q_{a_{\dag}} | p^{u^{*}}(a,t,x)= 0\}$. Take any function $w(a,t,x)\in L^{\infty}(Q_{a_{\dag}})$ such that $w(a,t,x)\neq 0$ for $(a,t,x) \in B$ and $w(a,t,x)\equiv 0$ for $(a,t,x) \in Q_{a_{\dag}}-B$ and $u^*+w\in U$. Let $z(a,t,x)=p^{u^*+w}-p^{u^*}$ and then it satisfies
\begin{equation*}
\left\{ \begin{array}{lll}
Dz- \delta \Delta z+\mu (a)z = u^*(a,t,x)z+w(a,t,x)z,&(a,t,x)\in Q_{a_\dagger}, \\
z(a,t,0) =z(a,t,24),&(a,t)\in (0,a_\dagger)\times (0,T), \\
\partial_x z(a,t,0) =\partial_x  z(a,t,24),&(a,t)\in (0,a_\dagger)\times (0,T), \\
\displaystyle z(0,t,x)=\displaystyle \int_0^{a_\dagger} \beta(a)\displaystyle \int_{x-\eta}^{x+\eta}K(x,s)z(a,t,s)ds da,&(t,x)\in (0,T)\times (0,24), \\
\displaystyle z(a,0,x)=0 , &(a,x)\in (0,{a_\dagger} ) \times (0,24).
\end{array}
\right.
\end{equation*}
By the uniqueness result, one can infer that $z(a,t,x)\equiv 0$ a.e. in $Q_{a_{\dagger}}$. This implies that we can change $u^*$ in $B$ with arbitrary values in $[\varsigma_1(a,t,x),\varsigma_2(a,t,x)]$ and the value of the related cost functional of $(OH)$ remains the same. Then the conclusion is obvious and the proof is complete.
\end{proof}

\appendix
\setcounter{secnumdepth}{0}
\section{Appendix}
\noindent
We present here a well-known result of Bellman in \cite{As}.

\begin{lemmaA}[Bellman]\label{lemma22}
If $x\in C([a,b])$, $\psi \in L^{1}(a,b)$, $\psi(t)\geq0 \ a.e. \ t \in(a,b)$, $M \in R$ and for each $t\in[a,b]$,
\[
x(t)\leq M+\displaystyle \int_a^{t}\psi(s)x(s)ds,
\]
then
\[
x(t)\leq M \exp\left(\displaystyle \int_a^{t}\psi(s)d s\right).
\]
\end{lemmaA}

\end{document}